\documentclass[11pt]{amsart}

\usepackage{fullpage}  

\usepackage[utf8]{inputenc}

\usepackage{amsmath}  
\usepackage{amsthm}  
\usepackage{amssymb}  
\usepackage{bm}  
\usepackage{mathtools}  

\usepackage{enumerate}

\usepackage{graphicx}  
\usepackage[dvipsnames]{xcolor}  
\usepackage[small,bf]{caption}  

\graphicspath{{figures/}}  
\graphicspath{{../figures/}}  

\usepackage{breakcites}  
\usepackage{url}  
\usepackage{hyperref}  

\hypersetup{colorlinks, citecolor=black, filecolor=black,
	linkcolor=black, urlcolor=black}

\usepackage{algorithm}
\usepackage[noend]{algpseudocode}

\usepackage{comment}  

\theoremstyle{plain}  
\newtheorem*{maintheorem}{Main Theorem}

\newtheorem{theorem}{Theorem}[section]
\newtheorem{proposition}[theorem]{Proposition}
\newtheorem{lemma}[theorem]{Lemma}

\theoremstyle{definition}

\newtheorem{exampleEnv}[theorem]{Example}
{\pushQED{\qed}\exampleEnv}
{\popQED\endexampleEnv}
\newenvironment{example*}{\exampleEnv}{\endexampleEnv}
\newtheorem{remarkEnv}[theorem]{Remark}
{\pushQED{\qed}\remarkEnv}
{\popQED\endremarkEnv}
\newenvironment{remark*}{\remarkEnv}{\endremarkEnv}

\def\gg{\mathfrak{g}}
\def\hh{\mathfrak{h}}
\def\mm{\mathfrak{m}}
\def\kk{\mathfrak{k}}

\numberwithin{equation}{section}

\DeclareMathOperator{\ric}{ric}
\DeclareMathOperator{\Ad}{Ad}
\DeclareMathOperator{\tr}{Tr}
\DeclareMathOperator{\ad}{ad}

\DeclareMathOperator{\Aut}{Aut}
\DeclareMathOperator{\SO}{SO}
\DeclareMathOperator{\SU}{SU}
\DeclareMathOperator{\Sym}{Sym}
\DeclareMathOperator{\Lie}{Lie}

\title{\bf On the Isometry Group of Immortal Homogeneous Ricci Flows}
\author{Roberto Araujo}
\address{WWU M\"unster, Mathematisches Institut\\
	Einsteinstr. 62\\
	48149 M\" unster\\
	Germany}
\email{r.araujo@uni-muenster.de}
\thanks{Funded by the Deutsche Forschungsgemeinschaft (DFG, German Research Foundation) under Germany’s Excellence Strategy EXC 2044–390685587, Mathematics Münster: Dynamics–Geometry–Structure and the CRC 1442 Geometry: Deformations and Rigidity}

\begin{document}
	
	\begin{abstract}
		We establish a structure result for the isometry group of non-compact, homogeneous manifolds admitting an immortal homogeneous Ricci flow solution.
	\end{abstract}

\maketitle

	\section{Introduction}
	
	A family of Riemannian metrics $g(t)$, $t \in [0,T)$, on a smooth manifold $M$ is a solution to the \textit{Ricci
	flow} starting at $g_0$ if it satisfies the following evolution equation introduced by Hamilton in \cite{ham}:
		\begin{equation} \label{eq:RF}
		\frac{\partial g(t)}{\partial t} = -2\ric(g(t)), \hspace{0.5cm} g(0)=g_0,
		\end{equation}
where $\ric(g)$ is the Ricci $(0,2)$-tensor of the metric $g$.

	A Ricci flow solution is called \textit{immortal} if it exists for all positive time $t$.  
	
	In a given homogeneous Riemannian manifold $\left(G/H,g_0\right)$, the Ricci flow equation \eqref{eq:RF} restricted to $G$-invariant metrics becomes a non-linear autonomous ordinary differential equation and thus there exists a unique solution $g(t)$ of $G$-invariant metrics on $G/H$ starting at $g_0$. 
	
	The homogeneous Ricci flow has been the object of research of many authors, particularly in low-dimensions. Since the seminal work of Isenberg-Jackson \cite{ij}, the Ricci flow on 3-dimensional homogeneous spaces has been thoroughly studied, as well as on some families of metrics on dimension 4 (cf. \cite{ijl}). We also refer the reader to \cite{lau}, where Lauret presents a general approach to the study of the homogeneous Ricci flow (which relies on the introduction of the bracket flow), and where one can also find many detailed examples of 3-dimensional homogeneous Ricci flow solutions. Nonetheless, in general, the long-time behavior of homogeneous Ricci flows is still not fully understood.
	
	By Bérard-Bergery's work in \cite{ber}, it has been known that a manifold admits a homogeneous positive scalar curvature metric if and only if its universal cover is not diffeomorphic to $\mathbb{R}^n$. Moreover, by work of Lafuente \cite{laf}, a homogeneous Ricci flow is immortal if and only if the scalar curvature is never positive throughout the flow. This means that for any initial homogeneous metric on $\mathbb{R}^n$ the Ricci flow is immortal. For a detailed qualitative analysis of many cases of immortal homogeneous Ricci flows up to dimension 4 we refer the reader to \cite{lot}.

	On the other hand, Böhm \cite{boe} showed that for any homogeneous metric on a compact manifold different from a flat torus the Ricci flow has finite extinction time. 
	
	Böhm and Lafuente proposed on \cite{bl18} the problem of showing that the universal cover of an immortal homogeneous Ricci flow solution is diffeomorphic to $\mathbb{R}^n$. This then got established as the dynamical Alekseevskii conjecture \cite{mfo22} after their solution of the long-standing Alekseevskii conjecture \cite{bl23}, which states an analogous claim for homogeneous Einstein manifolds of negative Ricci curvature. 
	
	In this article, we give a partial answer to the dynamical conjecture by providing the following structure result for the transitive Lie group $G$, when the flow is immortal.

\begin{maintheorem}\label{thm:main}
	Let $G/H$ be an almost-effective presentation of an immortal homogeneous Ricci flow solution $(M,g_t)$. Then $G$ has no normal, semisimple, compact subgroup $K \lhd G$.
\end{maintheorem}

	In other words, if a homogeneous manifold $M$ with an almost-effective presentation $G/H$ is such that $G$ has a non-trivial, connected, normal, semisimple, compact subgroup $K$, then for every initial $G$-invariant metric on $M$ the Ricci flow has finite extinction time. Up to a covering, if $G$ has a non-trivial, connected, normal, semisimple, compact subgroup $K$, then $G$ is a direct product Lie group $G=K \times \hat{G}$. Notice however, that this product may not be a Riemannian product and, thus, there will be in general non-trivial mixed Ricci terms. Nonetheless, our result does not require any assumption on this regard. We show that the homogeneous compact fibers $K/K\cap H$ on $M$ (which by assumption are not tori, hence admit scalar positive homogeneous metrics) shrink to a lower dimensional space in finite time.

	By a result of Dotti Miatello \cite{dot}, it was known that if a Ricci negative manifold has a transitive action by a unimodular group of isometries, then it actually has a transitive action by a semisimple group of isometries. The proof relies on an algebraic manipulation of the Ricci curvature formula for homogeneous spaces, yielding a non-negative lower bound.
	
	Jablonski-Petersen \cite{jp} expanded this result by showing the following. Let $G=G_{ss} \ltimes R$ be a Levi decomposition of the (not necessarily unimodular) transitive group of isometries of $(M,g)$, with $R$ being the radical and $G_{ss}$ being a semisimple Levi factor. If $G_c$ is the compact factor of $G_{ss}$ and $K$ is the kernel of the conjugation representation $C \colon G_c \to \Aut(R)$ restricted to $G_c$, then either the fiber $K/K\cap H$ of $M$ is $0$-dimensional or there is a Ricci non-negative direction tangent to $K/K\cap H$ \cite[Lemma 1.6]{jp}. In particular, this implies that if $(M,g)$ is Ricci negative with $G$ semisimple, then $G$ is of non-compact type \cite[Proposition 1.2]{jp}, since the radical and hence the representation $C$ is trivial. 
	
	On this paper, we present in Proposition \ref{prop:algebraic Bochner} a stronger version of \cite[Lemma 1.6]{jp} by giving an uniformly positive bound for the Ricci curvature on spaces that contradict the conclusion of the main theorem. Indeed, this estimate is the key tool to prove that under these conditions the Ricci flow has finite extinction time. 
	
	In terms of the dynamical Alekseevskii conjecture, our main result allows us, when considering immortal homogeneous Ricci flows, to restrict to presentations $M=G/H$ such that, given the Levi decomposition $G=G_{ss}\ltimes R$, the conjugation representation $C \colon G_c \to \Aut(R)$ is almost-faithful. \\

 	In Section 2, we briefly talk about reductive decompositions on Lie groups, so that we can write the Ricci tensor as an algebraic expression depending on the metric and the Lie brackets of the Lie algebra $\gg$ of $G$. And thus we can directly compute an algebraic Bochner theorem for the case in hand. In Section 3, we briefly introduce the homogeneous Ricci flow and then prove some short lemmas on Dini derivatives that we need in order to exploit the previously obtained algebraic estimates for the analysis of the Ricci flow. The main theorem is then an immediate consequence of the previous results. We then end that section with a brief discussion of the dynamical Alekseevskii conjecture in low dimensions. \\
 	
\textit{Acknowledgments.} I would like to thank my PhD advisor Christoph Böhm for his support and helpful suggestions. I would also like to
thank the referee for the useful comments.
	
	\section{Uniform lower bounds for the Ricci curvature}
	
	The proof of \cite[Lemma 1.6]{jp} is an indirect one, by applying Bochner's formula for the Laplacian of the norm squared of a Killing field, which yields via the maximum principle Bochner's theorem \cite{boc}, stating that a Ricci negative compact Riemannian manifold has discrete isometry group.
	
	Böhm gave a type of direct algebraic proof of Bochner's theorem \cite[Theorem 2.1]{boe} for non-flat compact homogeneous manifolds, showing an explicit preferred direction in which the curvature is Ricci positive. Indeed, this uniformly positive bound is what he exploits in order to conclude that any Ricci flow on a non-flat compact homogeneous manifold has finite extinction time \cite[Theorem 2.2]{boe}
	
	Our goal in this section is to follow a similar approach to the one in [Boe15] by showing on Proposition \ref{prop:algebraic Bochner} that in the case where $M=G/H$ is a non-compact almost-effective presentation, with $G$ containing a normal, semisimple, compact subgroup $K$,  there is a positive lower bound for the maximum of the Ricci curvature which is uniform on any $G$-invariant metric. This directly generalizes Jablonski-Petersen's result \cite[Lemma 1.6]{jp} in a way that allows us to understand the long-time behavior of the homogeneous Ricci flow in those cases.\\

	A Riemannian manifold $(M^n, g)$ is said to be homogeneous if its isometry group
	$I(M, g)$ acts transitively on $M$. If $M$ is connected (which we will assume from here onward unless otherwise stated), then each transitive, closed Lie subgroup $G <I(M, g)$ gives
	rise to a presentation of $(M, g)$ as a homogeneous space with a $G$-invariant metric
	$(G/H, g)$, where $H$ is the isotropy subgroup of $G$ fixing some point $p \in M$. We call this space a \textit{Riemannian homogeneous space}. 
	
	The $G$-action induces a Lie algebra homomorphism $\gg \to \mathfrak{X}(M)$ assigning to each $X \in \gg$ a Killing field on $(M, g)$, also denoted by X, and given by
	
	\[X(q) \coloneqq \left. \frac{d}{dt}\right|_{t=0} \exp (tX)\cdot q, \hspace{0.5cm} q \in M.\]
	
	If $\hh$ is the Lie algebra of the isotropy subgroup $H < G$ fixing $p \in M$, then it can be characterized as those $X \in \gg$ such that $X(p)=0$. Given that, we can take a complementary Ad($H$)-module $\mm$ to $\hh$ in $\gg$ and identify $\mm \cong T_pM$ via the above infinitesimal correspondence.
	
	In general, a homogeneous space $G/H$ is called \textit{reductive} if there exists a complementary vector space $\mm$ such that for the Lie algebras of $G$ and $H$, respectively $\gg$ and $\hh$, we have
	
	\begin{align*}
		\gg = \hh \oplus \mm \nonumber, \hspace{0.5cm} \Ad(H)(\mm) \subset \mm.
	\end{align*}
	
	This is always possible in the case of Riemannian homogeneous spaces, since, by a classic result on Riemannian geometry \cite[Chapter VIII, Lemma 4.2]{doC}, an isometry is uniquely determined by the image of the point $p$ and its derivative at $p$, hence the isotropy subgroup $H$ is a closed subgroup of $\SO(T_pM)$, and in particular it is compact. Indeed, if $H$ is compact, then one can average over an arbitrary inner product over $\gg$ to make it Ad($H$)-invariant and hence take $\mm \coloneqq \hh^{\perp}$. Given that, one can identify $\mm \cong T_{eH}G/H$ once and for all and with this identification there is a one-to-one correspondence between homogeneous metrics in $M \coloneqq G/H$, $p \cong eH$ and Ad($H$)-invariant inner products in $\mm$.

	Once we assumed that $G < I(M, g)$, by \cite[Chapter VIII, Lemma 4.2]{doC}, we have that $G/H$ is an \textit{effective presentation}, which means that the ineffective kernel given by $N\coloneqq\{h \in H \ | \ h\cdot q=q, \hspace{0.2cm}\forall q \in M\}=\bigcap_{g \in G} gHg^{-1}$ is trivial. In other words, this means that $G$ and $H$ have no non-trivial common normal subgroup.
	
	On the other hand, by the correspondence above, given a reductive decomposition $\gg=\hh\oplus\mm$, such that $\gg$ and $\hh$ have no non-trivial common ideal and an inner product $\langle \cdot, \cdot \rangle$ on $\mm$ such that ad($\hh$) is a subalgebra of skew-symmetric operators, we can reconstruct the Riemannian homogeneous space $(M,g)$ with \textit{almost-effective presentation} $M=\tilde{G}/\tilde{H}$. Namely, such that the ineffective kernel $\bigcap_{g \in \tilde{G}} g\tilde{H}g^{-1}$ is discrete, where $\tilde{G}$ and  $\tilde{H}<\tilde{G}$ are integral groups for the Lie algebras $\gg$ and $\hh$ respectively. In this manner we can restrict the problem to the Lie algebra level. Observe however that in this case $\tilde{H}$ is not necessarily compact, for example it could be the universal cover of a torus. It is clear that we can always assume that a presentation of a homogeneous manifold $M$ is effective or almost-effective, since we can quotient the transitive group $G$ by its ineffective kernel $N$, having then an effective presentation given by the still transitive action of $G/N$ on $M$, and analogously for the almost-effectiveness on the Lie algebra level.\\
	
	Let $(M,g)$ be a Riemannian homogeneous manifold with an almost-effective presentation $M=G/H$ and reductive decompostion $\gg = \hh \oplus \mm$ on a base point $p$, with the identification $\mm \cong T_pM$. The formula for the Ricci curvature of a homogeneous Riemannian manifold, \cite[Corollary 7.38]{be}, is given by 
	\begin{equation}\label{eq:ricci formula}
		\ric_g(X,X)= - \frac{1}{2}B(X,X) -\frac{1}{2}\sum_i \Vert[X,X_i]_{\mm}\Vert_g^2+ \frac{1}{4}\sum_{i,j}g([X_i,X_j]_{\mm},X)^2 - g([H_g,X]_{\mm},X).
	\end{equation}
	Here $[\cdot,\cdot]_\mm$ is the projection of the Lie brackets according to the reductive decomposition $\hh\oplus\mm$, $B$ is the Killing form, $\{X_i\}_{i=1}^n$ is an orthonormal basis of $\mm$, and $H_g$ is the mean curvature vector defined by $g(H_g,X) \coloneqq \tr(\ad_X)$.
	
	We are ready to prove the following proposition which is the geometrical essence of our main result.
	
	\begin{proposition} \label{prop:algebraic Bochner}
		Let $G/H$ be an almost-effective presentation of $(M,g)$. Assume there is a normal, semisimple, compact subgroup $K \lhd G$. Then every $G$-invariant metric on $G/H$ has a positive Ricci direction. 
	\end{proposition}  
	\begin{proof}
		Let $\gg$, $\kk$, and $\hh$ be Lie($G$), Lie($K$), and Lie($H$) respectively. 
		Since $K$ is a compact normal subgroup of $G$, then $C\coloneqq KH$ is a well-defined compact subgroup of $G$. 
		Let us fix an Ad($C$)-invariant background metric $\langle\cdot,\cdot\rangle$ on $\gg$. 
		
		 Now notice that since $\kk$ is an ideal of $\gg$, it is in particular an $\hh$-module, hence consider the following reductive decomposition of $\kk$ as an $\hh$-module
		
		\[	\kk= \mathfrak{k'} \oplus \mathfrak{m}_\kk,\]
	where $k'=\Lie(K\cap H)$ and $\mathfrak{m}_\kk \coloneqq \kk'^\perp \subset \kk$. By almost-effectiveness, we have that $\mm_\kk \neq \{0\}$. 
	
		Let moreover,
		
		\[	\hh = \mathfrak{k'} \oplus \mathfrak{h}_1\]
		be a decomposition of $\hh$ on $\hh$-submodules $\kk'$ and $\mathfrak{h}_1 \coloneqq\kk'^\perp \subset \hh$. 
		
		From this we get the following $\hh$-submodules direct sum decomposition for the $\hh$-module $\hh+\kk$,
		
		\[\hh+\kk =\mathfrak{h}_1\oplus \mathfrak{k'} \oplus \mathfrak{m}_\kk.\]
		Finally, let us take a complementary $\hh$-submodule $\mm'\coloneqq\left(\hh+\kk\right)^\perp$ to $\hh+\kk$ on $\gg$. Therefore,
		
	\[\gg = \mathfrak{h}_1\oplus \mathfrak{k'} \oplus \mathfrak{m}_\kk \oplus \mm',\]
		and we can do the identification $T_{eH}G/H \cong \mm \coloneqq \mathfrak{m}_\kk \oplus \mm'$.
		
		Let us write $g(\cdot,\cdot) = \langle P\cdot, \cdot\rangle$, where $P$ is an Ad($H$)-equivariant positive definite operator on $\mm$.
	 	If $Pr_{\mm_\kk}^{\perp}$ is the $\langle \cdot, \cdot \rangle$-orthogonal projection onto $\mm_\kk$, then the operator $P_{\mm_\kk} \coloneqq Pr_{\mm_\kk}^{\perp}\circ \left. P \right|_{\mm_{\kk}}$
	 	is also positive definite, hence we can diagonalize it. Let $\{\bar{U}_i\}_{i=1}^n$ be a diagonalizing $\langle\cdot,\cdot\rangle$-orthonormal frame for $P_{\mm_\kk}$, with respective eigenvalues $0<p_1 \leq \ldots \leq p_n $. 
	 	
	 	If we define $U_i \coloneqq\frac{\bar{U}_i}{\sqrt{p_i}}$ for $1 \leq i \leq n$, we get a $g$-orhtonormal frame for $\mm_{\kk}$, $\{U_i\}_{i=1}^n$, that we can complete with an arbitrary $g$-orthonormal complement $\{V_i\}_{i=1}^m$ to get a frame for $\mm$. Let $X \in \mm$, then the Ricci tensor formula \eqref{eq:ricci formula} gives us the following equality:
	
		\begin{align}\label{eq:ric for us}
			\ric_g(X,X) &= -\frac{1}{2}B(X,X) -\frac{1}{2}\sum_i \Vert[X,U_i]_{\mm}\Vert^2_g+\frac{1}{4}\sum _{i,j} g([U_i,U_j]_{\mm},X)^2-g([H_g,X]_{\mm},X) \\
			&+\frac{1}{2}(\sum _{i,j} g([U_i,V_j]_{\mm},X)^2-\sum_i \Vert[X,V_i]_{\mm}\Vert^2_g) +\frac{1}{4}\sum _{i,j} g([V_i,V_j]_{\mm},X)^2. \notag
		\end{align}
			
		Observe that by our choice of reductive decomposition of $\gg$, $\mm$ is not necessarily $\langle \cdot, \cdot \rangle$-orhtogonal to $\hh$. Nonetheless, we chose $\mm$ so that $\langle \kk', \mm \rangle=0$, hence the projection to $\mm$ with respect to the reductive decomposition $\gg=\hh\oplus\mm=\left(\mathfrak{h}_1\oplus \mathfrak{k'} \right)\oplus \left(\mathfrak{m}_\kk \oplus \mm'\right)$ restricted to $\kk=\kk'\oplus\mm_\kk$ coincides with $\left.Pr_{\mm_\kk}^{\perp}\right \vert_{\kk}$.
					
			Now let $X \in \mm_\kk$ be an eigenvector of $P_{\mm_\kk}$ with eigenvalue $p$. Since $[\gg,\kk] \subset \kk$ and $\langle \kk', \mm_\kk \rangle=0$, we get that
			
			\begin{align*}
				g([H_g,X]_{\mm},X)&=\langle Pr_{\mm_\kk}^{\perp}[H_g,X],P(X)\rangle = \langle [H_g,X],P_{\mm_\kk}(X)\rangle \\
				&=p\langle [H_g,X],X\rangle=-p\langle H_g, [X,X]\rangle=0,
			\end{align*}
		where in the second line we used that $\ad_X \in \ad(\kk) \subset \mathfrak{so}(\gg)$.
		
		Let us choose now $X_n \in \mm_\kk$ to be an eigenvector with respect to the largest eigenvalue $p_n$. Since $\{U_i\}_{i=1}^n$ and $X_n$ are in the ideal $\kk$ and $\ad(\kk) \subset \mathfrak{so}(\gg)$, by the same argument as above, we get that
	
		\begin{align*}
			\sum _{i,j} g([U_i,V_j]_{\mm},X_n)^2-\sum_i \Vert[X_n,V_i]_{\mm}\Vert^2_g &= \sum _{i,j} \langle ([U_i,V_j],P_{\mm_\kk}(X_n)\rangle ^2-\sum_{i,j} \langle [X_n,V_j], P_{\mm_\kk}(U_i) \rangle ^2\\
			&=p_n^2 \sum _{i,j} \langle [U_i,V_j],X_n\rangle ^2-p_i^2\sum_{i,j} \langle [X_n,V_j], U_i \rangle ^2 \\
			&=\sum _{i,j} \langle [U_i,X_n],V_j\rangle^2(p_n^2-p^2_i) \geq 0.
		\end{align*}
	
		If we consider the expression \eqref{eq:ric for us} for $\ric_g(X,X)$, we have thus shown that in the direction of an eigenvector $X_n \in \mm_\kk$ with respect to the largest eigenvalue $p_n$
		
		\begin{equation} \label{eq:ricci inequality}
			\ric_g(X_n,X_n) \geq -\frac{1}{2}B(X_n,X_n) -\frac{1}{2}\sum_i \Vert[X_n,U_i]_{\mm}\Vert^2_g+\frac{1}{4}\sum _{i,j} g([U_i,U_j]_{\mm},X_n)^2.
		\end{equation}
		
		Moreover, observe that since $\kk$ is an ideal of $\gg$, the Killing form of $\kk$ is the same as the Killing form of $\gg$ restricted to $\kk$ and $[X,U_i]_\mm = [X,U_i]_{\mm_\kk}$, where in the right-hand side we consider the projection of the brackets of $\kk$ with respect to the reductive decomposition $\kk = \kk'\oplus \mm_\kk$. 
		
		Thus, the right-hand side of the expression \eqref{eq:ricci inequality} is the Ricci tensor of $K/{K\cap H}$ with the metric induced by $\left.g\right|_{\mm_\kk}$ evaluated in the direction $X_n$ and by almost-effectiveness $\left(K/K\cap H,g\right)$ is not flat. Hence, using the estimate for $\ric_{K/K\cap H}$ obtained in the proof of  \cite[Theorem 3.1]{boe}, we get that
	
		\begin{align*}
			\ric_g(X,X) \geq \ric _{K/K\cap H}(X,X) \geq -\frac{B(X,X)}{4}>0,
		\end{align*} 
		where $X \in \mm_{\kk}$ is an eigenvector for the largest eigenvalue of $P_{\mm_{\kk}}$.
	\end{proof}

\begin{remark*}
Geometrically, the compact normal subgroup $K$ generates an extrinsically isometric foliation on $M$ by non-flat compact homogeneous manifolds. Then we can see that this lemma is a kind of algebraic Bochner theorem for this type of normal compact foliation. 
\end{remark*}

\section{Homogeneous Ricci Flow}

	In full generality, the Ricci flow \eqref{eq:RF} is a non-linear partial differential equation. In the case where $M$ is compact, Hamilton \cite{ham} proved short time existence and uniqueness of it. Later, Shi \cite{shi} showed that if $(M,g_0)$ is complete, noncompact, with bounded curvature, then the Ricci flow has a solution with bounded curvature on a short time interval and Chen and Zhu \cite{chenzu} proved the uniqueness of the flow within this class of complete and bounded curvature Riemannian metrics.
	
	Since every homogeneous manifold is complete and has bounded curvature, then there is a unique Ricci flow solution $g(t)$ which is complete with bounded curvature with an initial homogeneous metric $g_0$. From uniqueness it follows by diffeomorphism equivariance of the Ricci tensor that the Ricci flow preserves isometries \cite[Corollay 1.2]{chenzu}. Thus this complete bounded curvature Ricci flow solution $g(t)$ with an initial $G$-invariant metric $g_0$ remains $G$-invariant. In this case, we call the Ricci flow \textit{homogeneous} and the Ricci flow equation becomes the following autonomous non-linear ordinary differential equation on the space of $\Ad(H)$-invariant inner products on $\mm$
	
	\begin{align}\label{eq:ricci flow}
		\frac{dg(t)}{dt} &= -2\ric(g(t)), \hspace{0.5cm} g(0)=g_0.
	\end{align}
	The Ricci $(0,2)$-tensor in this case can be seen as the following smooth map 
	\[\ric: (\Sym^2(\mm))^H_+ \to (\Sym^2(\mm))^H,\]
	where $(\Sym^2(\mm))^H$ is the non-trivial vector space of Ad($H$)-invariant symmetric bilinear forms in $\mm$ and $(\Sym^2(\mm))^H_+$ the open set of positive definite ones.
	
	Note that by classical ODE theory, given an initial $G$-invariant metric $g_0$ corresponding to an initial $\Ad(H)$-invariant inner product, there is a unique $\Ad(H)$-invariant inner product solution corresponding to a unique family of $G$-invariant metrics $g(t)$ in $M$. And indeed one can use that to define the homogeneous Ricci flow on the locally homogeneous incomplete case (cf. \cite{bl18}). \\
	
	We will now establish some simple lemmas on analysis of metric spaces to be able to use our algebraic estimates even when the evolution of the largest (or smallest) eigenvalues above considered are not evolving smoothly.
	
	The next lemma is equivalent to the one in \cite[Lemma B.40]{chow}. The only difference is that there the authors proved a version for upper right-hand Dini derivatives and we use here for better convenience upper left-hand Dini derivatives. The following lemma can be used by us in a more straightforward way, so for the sake of completeness we prove it here.
	
	\begin{lemma} [Lemma B.40, \cite{chow}] \label{lemma:Dini}
		Let $C$ be a compact metric space, $I$ an interval of $\mathbb{R}$, and a function $g: C \times I \to \mathbb{R}$, such that $g$ and $\frac{\partial g}{\partial t}$ are continuous. Define $\phi \colon I \to \mathbb{R}$ by
		\[\phi(t) \coloneqq \sup_{x \in C}g(t,x)\]
		and its upper left-hand Dini derivative by
		\[\frac{d_L^+\phi(t) }{dt} \coloneqq \limsup_{h \to 0^+}\frac{\phi(t)-\phi(t-h)}{h}.\]
		
		Let $C_t \coloneqq \{x \in C \ | \ \phi(t)=g(t,x)\}$. We have that $\phi$ is continuous and that for any $t \in I$
		\[\frac{d_L^+\phi(t) }{dt} = \min_{x \in C_t}\frac{\partial g}{\partial t}(t,x).\]
	\end{lemma} 
	
	\begin{proof}
		Fix $t \in I$ and let $x \in C_t$. Consider $h \to 0^+$ and observe that  $\phi(t-h) \geq g(t-h,x)$ and $\phi(t)=g(t,x)$. Therefore
		
		\begin{align*}
			\limsup_{h \to 0^+}\frac{\phi(t)-\phi(t-h)}{h} \leq \limsup_{h \to 0^+}\frac{g(t,x)-g(t-h,x)}{h} = \frac{\partial g}{\partial t}(t,x).
		\end{align*}
		
		The fact that $\sup_{y \in C} g(t,y)$ is continuous is a well-known consequence of the continuity of $g$ in $t$. This in particular implies that if $t_n \to t$ and $x_n \in C_{t_n}$ is such that $x_n \to x$, then $x \in C_t$. Now let us consider $t-h_n$ with $h_n \to 0^+$ and $x_n \in C_{t-h_n}$. By the  compactness of $C$ we can (without loss of generality) pass to a subsequence and assume that $x_n \to x \in C_t$. For each $n \in \mathbb{N}$ we can apply the mean value theorem so that we get a $s_n \in (t-h_n, t)$ such that 
		
		\begin{align*}
			\limsup_{h_n \to 0^+}\frac{\phi(t)-\phi(t-h_n)}{h_n} &\geq \limsup_{h_n \to 0^+}\frac{g(t,x_n)-g(t-h_n,x_n)}{h_n} = \lim_{h_n \to 0}\frac{\partial g}{\partial t}(s_n,x_n)=\frac{\partial g}{\partial t}(t,x).
		\end{align*}
		
	\end{proof}
	
	 Lemma \ref{lemma:Dini} will be fundamental to us when combined with the following elementary real analysis result, which we couldn't find a reference for, but should be quite well-known. Again, for the sake of completeness we prove it here.
	
	\begin{lemma} \label{lemma:Dini Dynamic}
		Let $[a,b]$ be a closed interval of $\mathbb{R}$ and $f: [a,b] \to \mathbb{R}$ a continuous function such that $\frac{d^+f}{dt} \leq 0$ and $f(a)\leq0$. Then $f(t) \leq 0$.
	\end{lemma}
	
	\begin{proof}
		Suppose by contradiction that there exists $t_0$ with $f(t_0)>0$. Given $\epsilon>0$, consider the support function $g_{\epsilon}(t) = f(t_0) + \epsilon(t-t_0)$. Now observe that the closed set $I=\{t \in [a,t_0] \ | \ f(t)\geq g_{\epsilon}(t)\}$ is non-empty and let $t_1  = \inf\{t \ | \ t \in I\}$. Then the condition $\frac{d^+f}{dt} \leq 0$ for all $t \in [a,b]$ implies that there exists $\delta>0$ such that for all $t_1-\delta <t \leq t_1$,
	
		\begin{align*}
			f(t) \geq f(t_1) + \epsilon(t-t_1) \geq g_{\epsilon}(t_1) + \epsilon(t-t_1)=f(t_0) + \epsilon(t_1-t_0) +\epsilon(t-t_1)=g_{\epsilon}(t).
		\end{align*}
		This shows that $t_1=a$, hence $f(a) \geq f(t_0)+\epsilon(a-t_0)$. Since $\epsilon>0$ is arbitrary, this implies $0\geq f(a)\geq f(t_0)>0$, contradiction. \\
	\end{proof}
	
		We are now ready to prove our main result, which we restate here in other words for convenience.  
	
		\begin{theorem} [Main Theorem] \label{theo:main}
Let $G/H$ be an almost-effective presentation of $M$. Assume there is a normal, semisimple, compact subgroup $K \lhd G$. Then for every initial $G$-invariant metric $g_0$ on $M$ the Ricci flow solution $g(t)$ has finite extinction time.
	\end{theorem}

	\begin{proof}
			Let  $g(t)$ be an arbitrary $G$-invariant Ricci flow \eqref{eq:ricci flow} solution on $M$. Let $\mm_\kk \subset \mm$ be as in Proposition \ref{prop:algebraic Bochner}. Given a background metric $\langle \cdot, \cdot \rangle$ as in Proposition \ref{prop:algebraic Bochner}, there exists a $\langle \cdot, \cdot \rangle$-unitary vector $X=X_t \in \mathfrak{m}_\kk$ (which depends on the metric $g(t)$) such that $\ric_{g(t)}(X,X) \geq \frac{-B(X,X)}{4}$. 
			
		Applying Lemma \ref{lemma:Dini} to $C\coloneqq\{V \in \mm_\kk \hspace{0.1cm} | \hspace{0.1cm} \langle V, V \rangle = 1 \}$ and to the Ricci flow solution $g(t)$ restricted to $C$, we get the following estimate for $k(t)\coloneqq \sup_{V \in C} g(t)(V,V)$
			\[\frac{d_L^+k(t)}{dt} \stackrel{\ref{lemma:Dini}}{\leq} \frac{dg(t)}{dt}(X_t,X_t)=-2\ric_{g(t)}(X_t,X_t) \leq \frac{B(X_t,X_t)}{2}\leq \frac{b}{2} <0,\]
			where $b$ is the maximum of the Killing form restricted to the compact $C$, which is uniformly below zero.
			
			 This implies, by Lemma \ref {lemma:Dini Dynamic}, that $k(t)=\sup_{V \in C} g(t)(V,V)$ shrinks at least linearly, therefore the flow develops a singularity in finite time.\\
	\end{proof} 

\subsection{The dynamical Alekseevskii conjecture in low dimensions}
	
	We finish this section discussing the dynamical Alekseevskii conjecture in low dimensions. We show that, up to dimension 4, indeed the only simply-connected homogeneous manifolds admitting an immortal solution to the Ricci flow \eqref{eq:ricci flow} are the ones diffeomorphic to $\mathbb{R}^n$. For dimensions up to 3, homogeneous Ricci flows have been thoroughly studied (cf. \cite{ij} and Section 3.3 of \cite{lot}). In dimension 4, the long-time behavior of the homogeneous Ricci flow has been also previously studied in depth (Section 3.4 of \cite{lot} and \cite{ijl} have an analysis for the Ricci flow on some families of metrics). It follows by \cite[Theorem 3.1]{boe} that the conjecture up to this dimension is confirmed. In dimension 5, we can use Theorem \ref{theo:main} to confirm the conjecture in a large family of examples, but in this dimension it also occurs the first examples of non-compact, simply-connected, non-contractible homogeneous manifolds which lie outside the hypothesis of our main theorem. 
	
	By a classic structure theorem on Lie groups \cite[Theorem 14.3.11]{hn}, every simply-connected  Riemannian homogeneous space $\tilde{M}=G/H$ is diffeomorphic to 
	
	\[K/H \times \mathbb{R}^n,\]
	where $K/H$ is a simply-connected compact homogeneous space. 

	We are going to use the fact that the isotropy group $H$ at the point $p \in M$, is contained in $\SO(T_pM)$ and that $\dim I(M,g) \leq \frac{n(n+1)}{2}$, where $n=\dim M$. Moreover, equality holds if and only if $M$ is a space-form. Hence, we have only to consider $H$ a proper compact subgroup of $\SO(T_pM)$. Lifting this homogeneous geometry to the universal cover we may consider $\tilde{M}$ simply-connected, non-compact, and non-contractible, with an almost-effective presentation $\tilde{G}/\tilde{H}$, where $\tilde{G}$ is simply-connected. This means we can work on the Lie algebra level and restrict the analysis to the Lie algebras $\Lie(\tilde{G})=\gg$ and $\Lie(\tilde{H})=\hh$ respectively, where $\hh$ is a compactly embedded subalgebra of $\gg$, with no non-trivial common ideal with $\gg$.
	
	By the Levi decomposition \cite[Corolllary 5.6.8]{hn}, every Lie algebra $\gg$ can be written as the semi-direct product of its radical $\mathfrak{r}$, i.e. its maximal solvable ideal, and a semisimple subalgebra $\gg_{ss}$,
	
	\[\gg =\gg_{ss} \ltimes \mathfrak{r}.\]
	Moreover, we can assume, by \cite[Lemma 14.3.3]{hn}, that $\hh$ is contatined in a maximal compactly embedded subalgebra $\kk$ such that 
	
\[\kk= \left(\gg_{ss} \cap \kk\right) \oplus\left(\mathfrak{r} \cap \kk\right) \hspace{0.5cm} \text{and}\hspace{1cm}[\kk,\kk]\subset \gg_{ss}.\]

	In dimension 1 and 2, every simply-connected non-compact smooth manifold is contractible, so there is nothing to prove.
	
	In dimension 3, the only case to consider is when $\tilde{M} = \mathbb{S}^2 \times \mathbb{R}$. Indeed, since the smallest dimensional simple compact Lie algebra is $\mathfrak{su}(2)=\mathfrak{so}(3)$ which has dimension 3 and real rank $1$, the only possibility for $\tilde{M}$ non-contractible is $\hh = \mathfrak{so}(2)=\mathbb{R}$ and thus $\dim \tilde{G}= 4$. Therefore, we must have $\tilde{G}=\SU(2) \times \mathbb{R}$ and $H = \SO(2) \times \{e\}$.
	
	In dimension 4, the possibilities for the diffeomorphism types of $\tilde{M}$ non-compact and non-contractible are $\mathbb{S}^2 \times \mathbb{R}^2$ and $\mathbb{S}^3 \times \mathbb{R}$.
	
	If $\tilde{M}$ is diffeomorphic to $\mathbb{S}^3 \times \mathbb{R}$, then $\hh$ is $\{0\}$  or $\mathfrak{so}(3)$. In the former, $\gg = \mathfrak{su}(2) \oplus \mathbb{R}$ and in the latter $\gg = \mathfrak{su}(2) \oplus \mathfrak{su}(2) \oplus \mathbb{R}$. In the latter case, $\hh = \mathfrak{so}(3)=\mathfrak{su}(2)$ is a diagonal embedding on $\kk=\mathfrak{su}(2)\oplus \mathfrak{su}(2)$, and the geometry is that of the Riemannian product of a round $3$-sphere and a line. For the case $\hh=\{0\}$ the possible geometries include all the left-invariant metrics on $\SU(2)\times \mathbb{R}$ which include non-product metrics as well. Since $\mathfrak{su}(2)\oplus \mathbb{R}$ is a compact Lie algebra, the conjecture follows in this case by \cite[Theorem 3.1]{boe}.
	
	If $\hh =\mathfrak{so}(2)$ the only possibility is $\tilde{G} = \SU(2) \times \mathbb{R}^2$, and $H = \SO(2) \times \{e\}$. For $\hh = \mathfrak{so}(2)\oplus\mathfrak{so}(2)$, the possibilities are $\gg= \mathfrak{so}(3)\oplus \mathfrak{e}(2)$, where $\mathfrak{e}(2)=\mathfrak{so}(2)\ltimes \mathbb{R}^2$ is the Lie algebra of the group of isometries of the euclidean plane, or $\gg= \mathfrak{so}(3)\oplus \mathfrak{sl}(2,\mathbb{R})$. In any case, the diffeomorphism type of these spaces is given by $\mathbb{S}^2 \times \mathbb{R}^2$ and the geometry is actually a Riemannian product either of the round $2$-sphere with the euclidean plane $\mathbb{S}^2 \times \mathbb{E}^2$ or with the hyperbolic plane $\mathbb{S}^2 \times \mathbb{H}^2$. 
	
			It is clear that up to dimension 4 the dimension of $\gg$ is low enough that, for the purposes of checking the dynamical Alekseevskii conjecture, no non-trivial representation of semisimple Lie algebras appear in the Levi decomposition of the almost-effective presentations. In dimension 5, more interesting cases occur when $\tilde{M}$ is diffeomorphic to $\mathbb{S}^2\times \mathbb{R}^3$. We have for example the following families parameterized by pairs of coprime natural numbers $p$ and $q$, $\left(\mathfrak{e}(2) \oplus \mathfrak{su(2)}\right)/ \Delta_{p,q} \mathfrak{so}(2)$, and $\left(\mathfrak{sl}(2,\mathbb{R}) \oplus \mathfrak{su(2)}\right)/ \Delta_{p,q} \mathfrak{so}(2)$, both of which by Theorem \ref{theo:main} must satisfy the conjecture. 
			
			On the other hand, we also have in dimension 5 the lowest dimensional examples of non-compact, simply-connected, non-contractible $\tilde{M}$, with almost-effective presentation $\tilde{M}=G/H$ such that $G$ has no normal, semisimple, compact subgroup. Indeed, we have as noticeable examples $\mathfrak{e}(3)/\mathfrak{so}(2)$, where $\mathfrak{e}(3)= \mathfrak{so}(3)\ltimes \mathbb{R}^3$ is the Lie algebra of the group of isometries of the euclidean space, which is a semidirect product of a compact semisimple Lie algebra and an abelian radical, and also the case $\mathfrak{sl}(2,\mathbb{C})/ \mathfrak{u}(1)$, where the transitive group $G$ is simple of non-compact type.


\begin{thebibliography}{[CLN0600]}
		
		\bibitem[BB74]{ber}
		{\sc L. Bérard-Bergery}, \textit{Sur la courbure des métriques riemanniennes invariantes des groupes de Lie et des espaces homogènes}, Ann. Sci. Éc. Norm. Supér. (4), 11:4 (1978), 543-576.
		
		\bibitem[Bes87]{be}
		{\sc A. L. Besse}, \textit{Einstein manifolds}, Ergebnisse der Mathematik und ihrer Grenzgebiete (3) [Results in Mathematics and Related Areas (3)], vol. 10, Springer-Verlag, Berlin, 1987.
		
		\bibitem[BL18]{bl18}
		{\sc C. Böhm and R. A. Lafuente}, \textit{Immortal homogeneous Ricci flows}, Invent. Math. 212 (2018), no. 2, 461–529.
		
		\bibitem[BL23]{bl23}
		{\sc C. Böhm, R.  Lafuente}, \textit{Non-compact Einstein manifolds with symmetry}, J. Amer. Math. Soc. 36 (2023), no. 3, 591–651.
		
		\bibitem[Boe15]{boe}
		{\sc C. Böhm}, \textit{On the long time behavior of homogeneous Ricci flows}, Comment. Math. Helv. 90 (2015), no. 3, 543–571.
		
		\bibitem[Boc46]{boc}
		{\sc S. Bochner}, \textit{Vector fields and Ricci curvature}, Bull. Amer. Math. Soc. 52 (1946), 776–797. 
		
		\bibitem[CZ06]{chenzu}
		{\sc B.-L. Chen and X.-P. Zhu}, \textit{Uniqueness of the Ricci flow on complete noncompact
		manifolds}, J. Differential Geom. 74 (2006) 119–154.
		
		\bibitem[CLN06]{chow}
		{\sc B. Chow, P. Lu and L. Ni}, \textit{Hamilton’s Ricci flow}, Graduate Studies in Mathematics,
		77. American Mathematical Society, Providence, RI; Science Press, New York, 2006.
		
		\bibitem[doC92]{doC}
		{\sc M. do Carmo}, \textit{Riemannian Geometry}, Mathematics: Theory and Applications. Birkhäuser Boston, 1992.
		
		\bibitem[Dot88]{dot}
		 {\sc I. Dotti Miatello}, \textit{Transitive group actions and Ricci curvature properties}, Michigan Math. J. 35 (1988), no. 3, 427–434. 
		
		\bibitem[Ham82]{ham}
		{\sc R. Hamilton}, \textit{Three-manifolds with positive Ricci curvature}, J. Diff. Geom. 17 (1982), 255–306.
		
		\bibitem[HN11]{hn}
		{\sc J. Hilgert and K. Neeb}, \textit{Structure and geometry of Lie groups}, Springer Monographs in Mathematics, vol. 110, Springer, 2011.
		
		\bibitem[IJ92]{ij}
		{\sc J. Isenberg, M. Jackson}, \textit{Ricci flow of locally homogeneous geometries on closed manifolds}, J. Differ. Geom. 35 (1992), 723–741.
		
		\bibitem[IJL06]{ijl}
		{\sc J. Isenberg, M. Jackson, P. Lu}, \textit{Ricci flow on locally homogeneous closed 4-manifolds}, Commun. Anal. Geom. 14 (2006), 345–386.
		
		\bibitem[JP17]{jp}
		{\sc M. Jablonski, P. Petersen}, \textit{A step towards the Alekseevskii Conjecture}, Math. Ann. 368 (2017), 197-212.
		
		\bibitem[Laf15]{laf}
		{\sc R. Lafuente}, \textit{Scalar curvature behavior of homogeneous Ricci 
			flows}, The Journal of Geometric Analysis, vol. 25 iss. 4 (2015), 2313-2322. 
		
		\bibitem[Lau13]{lau}
		{\sc J. Lauret}, \textit{Ricci flow of homogeneous manifolds}, Math. Z. 274 (2013), 373–403.
		
		\bibitem[Lot07]{lot}
		{\sc J. Lott}, \textit{On the long-time behavior of type-III Ricci flow solutions}, Math. Ann. 339 (2007), 627–666.
		
		\bibitem[MFO22]{mfo22}
		{\sc A. Naber, A. Neves, B. Wilking}, \textit{Geometrie}. Oberwolfach Rep. 19 (2022), no. 2, pp. 1551–1601
		
		\bibitem[Shi89]{shi}
		{\sc W.X. Shi}, \textit{Deforming the metric on complete Riemannian manifolds}, J. Differential Geometry 30 (1989) 223–301.
		
	\end{thebibliography}
\end{document}